\documentclass[12pt,reqno]{amsart}

\usepackage{mathrsfs}

\usepackage[margin=2cm]{geometry}

\usepackage{enumerate}
\usepackage{amsthm}
\usepackage{amsmath}
\usepackage{amsfonts}
\usepackage{amssymb}
\usepackage[nobysame, initials,logical-quotes]{amsrefs}
\usepackage[ocgcolorlinks,hyperfootnotes=false,colorlinks=true,citecolor=blue,linkcolor=blue,urlcolor=blue]{hyperref}
\usepackage{graphicx}
\usepackage{color}
\usepackage{graphics}
\usepackage{eepic}
\usepackage{nicefrac}
\usepackage{tikzsymbols}
\newcommand{\ignore}[1]{}

\usepackage[colorinlistoftodos]{todonotes}

\newcommand{\bydef}{\stackrel{\rm def}{=}}

\usepackage{mathtools}
\mathtoolsset{showonlyrefs,showmanualtags}

\newcommand{\cP}{{\mathcal P}}

\newcommand{\bD}{{\mathbb D}}
\newcommand{\bC}{{\mathbb C}}

\newtheorem{thm}{Theorem}[section]

\newtheorem{lemma}[thm]{Lemma}

\newtheorem{proposition}[thm]{Proposition}

\theoremstyle{definition}

\newtheorem{example}[thm]{Example}

\numberwithin{equation}{section}

\def\textmatrix#1&#2\\#3&#4\\{\bigl({#1 \atop #3}\ {#2 \atop #4}\bigr)}
\def\dispmatrix#1&#2\\#3&#4\\{\left({#1 \atop #3}\ {#2 \atop #4}\right)}
\numberwithin{equation}{section}

\def\textmatrix#1&#2\\#3&#4\\{\bigl({#1 \atop #3}\ {#2 \atop #4}\bigr)}
\def\dispmatrix#1&#2\\#3&#4\\{\left({#1 \atop #3}\ {#2 \atop #4}\right)}
\begin{document}
	
	\title{Rational Penta-inner functions and the distinguished boundary of the pentablock}
	\author[Jindal]{Abhay Jindal}
	\address[Jindal]{Department of Mathematics, Indian Institute of Science, Bengaluru-560012, India}
	\email{abjayj@iisc.ac.in}
	
	\author[Kumar]{Poornendu Kumar}
	\address[Kumar]{Department of Mathematics, Indian Institute of Science, Bengaluru-560012, India}
	\email{poornenduk@iisc.ac.in}
	
	\thanks{
		{\em Key words and phrase}: Rational inner functions, symmetrized bidisc, pentablock, distinguished
		boundary.}

	\maketitle
	
	\begin{abstract}
		
		In this note, we give a description of rational maps from the open unit disc $\mathbb{D}$ to the pentablock that map the boundary of $\mathbb{D}$ to the distinguished boundary of the pentablock. We also obtain a new characterization of the distinguished boundary of the pentablock.
	\end{abstract}
	\section{Introduction}
	In 2015, Agler, Lykova and Young introduced a new bounded domain called pentablock in \cite{ALY-JMAA}. The pentablock is a subdomain of $\mathbb{C}^{3}$ denoted by $\cP$ and defined as the image of the domain $\{A \in M_{2}(\mathbb{C}): \|A\| < 1\}$ under the mapping 
	$$ \pi: A = [a_{ij}]\mapsto (a_{21}, \operatorname{tr}(A), \operatorname{det}(A)).$$
	We denote the closure of $\mathcal P$ by $\overline{\mathcal P}$. The set $\cP \subset \mathbb{C}^{3}$ is non-convex, polynomially convex, and star-like about the origin, see \cite{ALY-JMAA}. The pentablock is an inhomogeneous domain, see \cite{LK}. The complex geometry and function theory of the pentablock were further developed in \cite{ALY-JMAA, LK, GS, PZ}.
	
	Attempts to solve particular cases of the $\mu-$synthesis problem have also led to the study of two other domains namely the {\em symmetrized bidisc} 
	$$\mathbb{G} := \{ (\operatorname{tr}(A), \operatorname{det}(A)) : A = [a_{ij}]_{2 \times 2}, \|A\| <1\}\subset\bC^2,$$
	 see \cite{AY} and the {\em tetrablock} 
	 $$\mathbb{E}:= \{ (a_{11},a_{22}, \operatorname{det}(A)) : A = [a_{ij}]_{2\times 2}, \|A\| < 1\} \subset \mathbb{C}^{3},$$
	 see \cite{AWY}. We denote the closure of $\mathbb{G}$ by ${\Gamma}.$
	The set $\mathbb{G}$ and $\mathbb{E}$ are polynomially convex and non-convex domains. The symmetrized bidisc and the tetrablock have attracted a considerable amount of interest in recent years. For a greater exposition on these domains, see \cite{AWY, AYTran, ALY-PLMS, AY, Ball-Sau, Tirtha-Tetrablock, TPS-Adv, DKS, Jaydeb, Hari, Lempert-Tetra}.

		Let $\Omega \subset \mathbb{C}^{d}$ be a bounded polynomially convex domain with closure $\overline{\Omega}.$ Let $A(\Omega)$ be the algebra of continuous scalar functions on $\overline{\Omega}$ that are holomorphic in $\Omega.$ A {\em boundary} for $\Omega$ is a subset $C$ of $\overline{\Omega}$ such that every function in $A(\Omega)$ attains its maximum modulus on $C.$ The {\em distinguished boundary} of $\Omega,$ to be denoted by $b\Omega$ (some authors write $b \overline{\Omega}$), is the smallest closed boundary of $\Omega.$ 
	
	The distinguished boundaries of the symmetrized bidisc and the tetrablock were found in \cite{AY} and \cite{AWY} to be 
	\begin{align*}
	b \Gamma & =  \{(s,p)\in \mathbb{C}^{2}: |s| \leq 2, s = \overline{s} p, |p| =1\}\\
	& = \{ (\operatorname{tr}(U), \operatorname{det}(U)) : U = [u_{ij}]_{2 \times 2}, U \text{ is a unitary}\}
	\end{align*}
	and 
	$$b \mathbb{E} = \{ (u_{11}, u_{22}, \operatorname{det}(U)) : U = [u_{ij}]_{2 \times 2}, U \text{ is a unitary}\},$$
	respectively. A key fact used in the above descriptions of distinguished boundaries is that the set of $2 \times 2$ unitary matrices is the distinguished boundary of the $2 \times 2$ matrix operator-norm unit ball. It was shown in reference \cite{ALY-JMAA} that the sets 
	$$K_{0} = \biggl\{ (a,s,p) \in \mathbb{C}^{3}: (s,p)\in b \Gamma, |a| = \sqrt{1 - \frac{1}{4}|s|^{2}}  \biggr\}$$ 
	and 
	$$K_{1} = \biggl\{ (a,s,p) \in \mathbb{C}^{3}: (s,p)\in b \Gamma, |a| \leq  \sqrt{1 - \frac{1}{4}|s|^{2}}  \biggr\}$$ 
	both are boundaries of the pentablock. It was further shown in reference \cite{ALY-JMAA} that the set $K_{0}$ is the distinguished boundary of the pentablock while
	$$K_{1} = \{(u_{21}, \operatorname{tr}(U), \operatorname{det}(U)) : U = [u_{ij}]_{2 \times 2}, U \text{ is a unitary}\}.$$
	This suggests that, unlike in the cases of the symmetrized bidisc and tetrablock, the distinguished boundary of the pentablock is attuned to a certain special class of unitary matrices rather than whole class of unitary matrices. This note finds exactly that special class that describes $K_{0}$ via the map $\pi.$ This, in turn, leads to a new description of the distinguished boundary of the pentablock.    
	
   Let $\mathbb{T}$ denote the unit circle in the complex plane $\mathbb{C}$. An analytic map $x = (x_{1},\dots,x_{d}):\mathbb{D}\rightarrow\overline{\Omega}$ is called a {\em rational $\Omega-$inner} (some authors call it {\em rational $\overline{\Omega}-$inner}) function if each $x_{i}$ is a rational function with poles outside $\overline{\mathbb{D}}$ and $$(x_{1}(\lambda), \dots, x_{d}(\lambda)) \in b\Omega$$ 
   for all $\lambda \in \mathbb{T}.$ In \cite{Blachke},  W. Blaschke  studied the rational $\mathbb{D}-$inner functions and proved that all rational $\mathbb{D}-$inner functions are of the form
	\begin{align}\label{Blas}
		B(z):=e^{i\theta}\prod_{j=1}^n\frac{z-a_j}{1-\overline{a_j}z}
	\end{align}
	for some $a_1,a_2,...,a_n\in\mathbb{D}$ and $\theta\in[0, 2\pi]$. Functions
	of this form are well-known to be the finite Blaschke product. For a survey of results, see \cite{CGP}. If $\Omega=\mathbb{D}^d$, then it follows from $d=1$ case that all rational $\mathbb
		{D}^d-$inner functions are of the form
	$$(B_1(z), \dots B_d(z))$$
	for some finite Blaschke products $B_1, \dots, B_d$. A description of rational $\Gamma-$inner functions is given by Agler--Lykova--Young, see \cite{ALY-Adv}. Alsalhi--Lykova gave a description of rational ${\mathbb{E}}-$inner functions, see \cite{Omar-Lykova}. In section 3, we give a description of rational $\cP-$inner functions, see Theorem \ref{Penta-inner}.

	Sometime after this paper was finished and uploaded to arXiv, \cite{NZ} appeared on arXiv. There is an overlap of one result of our paper with \cite{NZ}. Theorem \ref{Penta-inner} also appears there. The proofs are different.  Fej\'er-Riesz Theorem is used in \cite{NZ} whereas our proof uses a study of the zeros and poles of certain functions.

	\section{A new characterization of the distinguished boundary}
	
	In the following theorem, we shall give a characterization of points in $b \cP.$ The proof of the theorem will manifest a recipe to construct a $2 \times 2$ unitary matrix $U = [u_{ij}]$ for any $(a,s,p) \in b \cP$ such that $(a,s,p) = (u_{21}, \operatorname{tr}(U), \operatorname{det}(U)).$
	\begin{thm}\label{dist}
		For $(a,s,p)\in\mathbb{C}^{3},$ the following are equivalent:
		\begin{enumerate}
			\item $(a,s,p)\in b \cP$,
			\item There exists a unique unitary matrix $U = \begin{pmatrix}
				u_{11} & u_{12} \\ u_{21} & u_{22}
			\end{pmatrix}$ such that
		$$u_{11} = u_{22} \hspace{5mm} \text{and} \hspace{5mm} (a,s,p) = (u_{21}, \operatorname{tr}(U), \operatorname{det}(U)).$$
		\end{enumerate}
	\end{thm}
	
	\begin{proof}
		First, we shall prove that (1) $\Rightarrow$ (2). Let $(a,s,p)\in b \cP.$ Since $b \cP = K_{0},$ we have
		$$ |s| \leq 2, \hspace{3mm} s = \overline{s} p,\hspace{3mm} |p|=1\quad \text{ and }\quad |a|^{2} = 1- \frac{|s|^{2}}{4}.$$
		In order to find the desired matrix $U = [u_{ij}]_{2 \times 2}$, we need to solve the following four equations in four variables.
		$$u_{11} - u_{22}  = 0, \hspace{3mm} u_{21} = a, \hspace{3mm} u_{11} +u_{22} = s\quad \text{ and } \quad u_{11}u_{22} - u_{12} u_{21} = p.$$
		If $a \neq 0,$ then we get a unique solution 
		$$(u_{11},u_{12},u_{21},u_{22}) = \left(\frac{s}{2}, \frac{s^{2} - 4p}{4a}, a, \frac{s}{2} \right).$$
		A simple computation will show that the matrix $ U$ is unitary. If $a=0,$ then the set of solutions is 
		$$\{ (u_{11}, u_{12}, u_{21}, u_{22}) = (\frac{s}{2}, \lambda, 0, \frac{s}{2}): \lambda\in \mathbb{C}, s^{2} = 4p\}.$$
		Since $|p| =1,$ we get $|s| =2$ and hence the matrix $U = [u_{ij}]_{2 \times 2}$ is unitary if and only if $\lambda =0.$
		
		Now we shall prove that (2) $\Rightarrow$ (1). Let $U = [u_{ij}]_{2 \times 2} $  be a unitary matrix with   
		$$u_{11} = u_{22} \hspace{5mm} \text{and} \hspace{5mm} (a,s,p) = (u_{21}, \operatorname{tr}(U), \operatorname{det}(U)).$$
		Since $U$ is a unitary, we get that $(s,p) = (\operatorname{tr}(U) , \operatorname{det}(U)) \in b \Gamma,$ also
		$$	4 |a|^{2} + |s|^{2} = 4 |u_{21}|^{2} + |\operatorname{tr}(U)|^{2} = 4(|u_{21}|^{2} + |u_{11}|^{2}) = 4.$$
		This proves that $(a,s,p) \in b \cP.$
	\end{proof}

	\section{Rational $\cP-$inner functions}
	In this section, we give a description of rational $\cP-$inner functions. First, recall that,
	a rational map $x=(x_1, x_2, x_3):{\mathbb{D}}\rightarrow\overline{\cP}$ is said to be {\em rational $\cP-$inner} if
	$$\big(x_1(\lambda), x_2(\lambda), x_3(\lambda)\big)\in b \cP$$
	for all $\lambda\in\mathbb{T}$. Note that if $(s,p)\in \Gamma $ and $\alpha\in\bD,$ then $1 - s \alpha + p \alpha^{2} \neq 0,$ see \cite{AY-model}. For each $\alpha\in\bD,$ we define a function $\Psi_{\alpha}: \mathbb{C} \times \Gamma \to \mathbb{C}$ by $$\Psi_{\alpha}(a,s,p) = \frac{a(1 - |\alpha|^{2})}{1 - s \alpha + p \alpha^{2}}.$$ The function $\Psi_{\alpha}$ is analytic in $\mathbb{C} \times \mathbb{G}$ and continuous on $\mathbb{C} \times \Gamma.$ One of the main results of \cite{ALY-JMAA} contains several characterization of a point to be in $\overline{\cP}$. We recall the one characterization which we shall use later.
	\begin{thm}\cite[Theorem 5.3]{ALY-JMAA}\label{point}
		For $(a,s,p)\in \mathbb{C} \times \Gamma,$ the following are equivalent:
		\begin{enumerate}
			\item $(a,s,p) \in \overline{\cP},$
			\item $|\Psi_\alpha(a,s,p)| \leq 1$ for all $\alpha \in \bD.$
		\end{enumerate}
	\end{thm}
	For any positive integer $n$ and for any polynomial $f$ of degree less than or equal to $n$, we define the polynomial $f^{\sim n}$ by the formula,
	$$f^{\sim n}(\lambda)=\lambda^n\overline{f\left(\frac{1}{\overline{\lambda}}\right)}.$$
	For a $\mathbb{C}-$valued rational function $x = f/g,$ where $f$ and $g$ are relatively prime polynomials, we define $\operatorname{deg}(x)$ to be the maximum of $\operatorname{deg}(f), \operatorname{deg}(g).$ Note that if $x$ is a finite Blashcke product, then $\operatorname{deg}(x)$ is same as number of Blaschke factors in the product. The following theorem gives a description of rational $\Gamma-$inner functions.
	\begin{thm} \cite[Proposition 2.2]{ALY-Adv}\label{Gamma-Inner}
		Let $h=(s, p)$ be a rational $\Gamma-$inner function with $\operatorname{deg}(p) =n$. Then  there exist polynomials $D$
		and $N$ such that
		\begin{enumerate}
			\item $\operatorname{deg}(D)$, $\operatorname{deg}(N)\leq{n}$
			\item $N^{\sim n}(\lambda)=N(\lambda)$ on $\overline{\mathbb{D}}$,
			\item $D(\lambda)\neq{0}$ on $\overline{\mathbb{D}}$,
			\item $|N(\lambda)|\leq 2|D(\lambda)|$ on $\overline{\mathbb{D}}$,
			\item $s=\frac{N}{D}$ on $\overline{\mathbb{D}}$, and
			\item $p=\frac{D^{\sim n}}{D}$ on $\overline{\mathbb{D}}$.
		\end{enumerate}
		Conversely, if $N$ and $D$ are polynomials satisfying $(1),(2)$ and $(4)$ above, $D(\lambda) \neq 0$ on $\mathbb{D},$ and $s$ and $p$ are defined by $(5)$ and $(6)$ respectively, then $h = (s, p)$ is a rational $\Gamma-$inner function with $\operatorname{deg}(p) = n$.
		
		Furthermore, a pair of polynomials $N'$ and $D'$ satisfies $(1)–(6)$ if and only if there
		exists a non-zero real number $t$ such that $N=tN'$ and $D=tD'.$
	\end{thm}
	
   Note that if $x = (x_{1}, x_{2}, x_{3})$ is a rational $\cP-$inner function, then in particular, 
   \begin{enumerate}
   \item $(x_{2}(\lambda), x_{3}(\lambda)) \in \mathbb{G}$ for every $\lambda \in \mathbb{D};$ and 
   \item $(x_{2}(\lambda), x_{3}(\lambda)) \in b \Gamma$ for every $\lambda \in \mathbb{T}.$
   \end{enumerate}
  Consequently, it is necessary for $x = (x_{1}, x_{2}, x_{3})$ to be rational $\cP-$inner that $(x_{2}, x_{3})$ be $\Gamma-$inner. The latter class is completely understood in view of Theorem \ref{Gamma-Inner}. Thus, our job reduces to understanding just the first coordinate of a rational $\cP$-inner function. This is what we do in the following sequence of preliminary results.

	\begin{lemma}\label{suff}
		If $(x_{2},x_{3})$ is a rational $\Gamma-$inner function and $x_{1}$ is a rational function with poles outside $\overline{\mathbb{D}}$ such that
		$$|x_{1}(\lambda)|^{2} = 1 - \frac{|x_{2}(\lambda)|^{2}}{4}$$
		for all $\lambda \in \mathbb{T},$ then $x = (x_{1},x_{2},x_{3})$ is a rational $\cP-$inner function.
	\end{lemma}
	\begin{proof}
		First note that $x(\lambda) = \big(x_{1}(\lambda), x_{2}(\lambda), x_{3}(\lambda)\big) \in b \cP$ for all $\lambda\in\mathbb{T}.$ We need to show that $\big(x_{1}(\lambda), x_{2}(\lambda), x_{3}(\lambda)\big) \in \overline{\cP}$ for all $\lambda \in \bD.$ Fix  $\alpha\in\bD$ and consider the map $\Psi_{\alpha} \circ x: \overline{\bD} \to \mathbb{C}. $ The map $\Psi_{\alpha} \circ x$ is analytic in $\bD$ and continuous on $\overline{\bD}.$ Since $x(\lambda)\in b \cP \subset \overline{\cP}$ for $\lambda\in\mathbb{T},$ by Theorem \ref{point}, for all $\lambda\in\mathbb{T}$ we get
		$$|\Psi_{\alpha}\big(x(\lambda)\big)| = |\Psi_{\alpha}\big(x_{1}(\lambda), x_{2}(\lambda), x_{3}(\lambda)\big)| \leq 1$$
		for all $\alpha \in \bD.$ By the maximum modulus principle, for $\lambda \in \bD$ we get
		$$|\Psi_{\alpha}\big(x(\lambda)\big)| = |\Psi_{\alpha}\big(x_{1}(\lambda), x_{2}(\lambda), x_{3}(\lambda)\big)| \leq 1$$
		for all $\alpha \in \bD.$ Again by Theorem \ref{point}, $x(\lambda) = \big(x_{1}(\lambda), x_{2}(\lambda), x_{3}(\lambda)\big) \in \overline{\cP}$ for all $\lambda \in \overline{\bD}.$ Thus, $x = (x_{1},x_{2},x_{3})$ is a rational map from $\bD$ to $\overline{\cP}$ which sends $\mathbb{T}$ into $b \cP.$ This proves that $x = (x_{1},x_{2},x_{3})$ is a rational $ \cP-$inner function.
		
	\end{proof}
	Now, we shall give some examples of rational $\cP-$inner functions.
	\begin{example} Let $B$ be a finite Blaschke product. Then the function $x:\mathbb{D}:\rightarrow\overline{\cP}$ defined by
		$$x(\lambda)=\big(B(\lambda), 0, B(\lambda)\big)$$ is rational $\cP-$inner.
	\end{example}
	\begin{proof}
		It is easy to see that $(0, B(\lambda))$ is a rational $\Gamma-$inner function. Now we show that, for $\lambda\in\mathbb{T}$, the point $x(\lambda)$ lies in $b \cP$. Here
		$$x_1(\lambda)=B(\lambda), \quad x_2(\lambda)= 0, \quad \text{and} \quad  x_3(\lambda)=B(\lambda).$$
		Since $|B(\lambda)|=1$ on the circle, it follows that
		$$|x_{1}(\lambda)|^2= 1=1-\frac{|x_2(\lambda)|^2}{4}.$$
		Thus, by Lemma \ref{suff}, $x$ is a rational $\cP-$inner function.
	\end{proof}
	The following lemma gives a class of rational $\cP-$inner functions.
	\begin{lemma}
		Let $\beta\in\mathbb{T}$. Then the map $x:\mathbb{D}\rightarrow\overline{\cP}$ by the setting $$\lambda\mapsto\left(\frac{\beta-\overline{\beta}\lambda}{2}, \beta+\overline{\beta}\lambda, \lambda\right)$$ is rational $\cP-$inner.
	\end{lemma}
	\begin{proof} By virtue of Lemma \ref{suff}, we need to show that $(x_2, x_3)$ is a $\Gamma-$inner function, and  the following equality holds for $\lambda\in\mathbb{T}$,
		$$4|x_1(\lambda)|^2+|x_2(\lambda)|^2=4.$$
		Here,
		$$x_1(\lambda)= \frac{\beta-\overline{\beta}\lambda}{2}, \quad x_2(\lambda)=\beta+\overline{\beta}\lambda \quad\text{ and } \quad x_3(\lambda)=\lambda.$$
		Note that, for $\lambda\in\mathbb{T}$,  $x_{2}(\lambda) = \overline{x_{2}(\lambda)} x_{3}(\lambda)$, $|x_3(\lambda)|=1,$  and $|x_{2}(\lambda)| \leq 2.$ So the map $(x_{2},x_{3})$ maps $\mathbb{T}$ into $b \Gamma.$ Since $(x_{2}(\lambda), x_{3}(\lambda)) \in \Gamma$ for all $\lambda\in\bD,$ it follows that $(x_2, x_3)$ is a rational $\Gamma-$inner function. Now, for $\lambda\in\mathbb{T}$,
		\begin{align}\label{PX1}
			|x_1(\lambda)|^2&=x_1(\lambda)\overline{x_1(\lambda)}=1/4(\beta-\overline{\beta}\lambda)(\overline{\beta}-\beta\overline{\lambda})\nonumber\\
			&=\frac{1}{4}\left[|\beta|^2-\overline{\beta}^2\lambda-\beta^2\overline{\lambda}+|\beta|^2|\lambda|^2 \right]\nonumber\\
			&=\frac{1}{2}-\frac{1}{4}\left[\overline{\beta}^2\lambda+\beta^2\overline{\lambda} \right].
		\end{align}
		We also have
		\begin{align}\label{PX2}
			|x_2(\lambda)|^2&=x_2(\lambda)\overline{x_2(\lambda)}= (\beta+\overline{\beta}\lambda)(\overline{\beta}+\beta\overline{\lambda})\nonumber\\
			&=|\beta|^2+\beta^2\overline{\lambda}+\overline{\beta}^2\lambda+|\beta|^2|\lambda|^2\nonumber\\
			&=2+\beta^2\overline{\lambda}+\overline{\beta}^2\lambda
		\end{align}
		Thus, from equations \eqref{PX1} and \eqref{PX2}, for all $\lambda\in\mathbb{T}$,
		$$4|x_1(\lambda)|^2+|x_2(\lambda)|^2=4.$$
	\end{proof}
	The next two lemmas give some more examples of rational $\cP-$inner functions. These will also be used in the proof of the main theorem of this section.
	\begin{lemma}
		If $x = (x_{1},x_{2},x_{3})$ is a rational $\cP-$inner function, then $x_{B} \bydef (Bx_{1}, x_{2}, x_{3})$ is also a rational $\cP-$inner function for any finite Blaschke product $B.$
	\end{lemma}
	\begin{proof}
		Since $(x_{1},x_{2},x_{3})$ is a rational $\cP-$inner function, $(x_{2},x_{3})$ is a $\Gamma-$inner function. For $\lambda\in\mathbb{T},$
		\begin{align*}
			4 |Bx_{1}(\lambda)|^{2} + |x_{2}(\lambda)|^{2}
			& = 4 |B(\lambda)|^{2} |x_{1}(\lambda)|^{2} + |x_{2}(\lambda)|^{2} \\
			& = 4|x_{1}(\lambda)|^{2} + |x_{2}(\lambda)|^{2}\\
			& = 4.
		\end{align*}
		Thus, by Lemma \ref{suff}, $x_{B} = (Bx_{1},x_{2},x_{3})$ is a rational $\cP-$inner function.
	\end{proof}
	\begin{lemma}
		If $B$ is a finite Blaschke product, $x_{1}$ is a rational function with poles outside $\overline{\mathbb{D}}$ and $(Bx_{1}, x_{2}, x_{3})$ is a rational $\cP-$inner function, then $(x_{1},x_{2},x_{3})$ is also a rational $\cP-$inner function.
	\end{lemma}
	\begin{proof}
		Since $(Bx_{1},x_{2},x_{3})$ is a rational $\cP-$inner function, $(x_{2},x_{3})$ is a $\Gamma-$inner function. For $\lambda \in \mathbb{T},$
		\begin{align*}
			4 |x_{1}(\lambda)|^{2} + |x_{2}(\lambda)|^{2}
			& = 4 |B(\lambda)|^{2} |x_{1}(\lambda)|^{2} + |x_{2}(\lambda)|^{2}\\
			& = 4 |Bx_{1}(\lambda)|^{2} + |x_{2}(\lambda)|^{2} \\
			& = 4.
		\end{align*}
		Thus, by Lemma \ref{suff}, $(x_{1},x_{2},x_{3})$ is a rational $\cP-$inner function.
	\end{proof}
	If $f(z) = \sum\limits_{i=1}^{n} a_{i}z^{i}$ is a polynomial, then define $$f^{\vee} (z) = \sum\limits_{i=1}^{n} \overline{a_{i}} z^{i}.$$ If $f_{1},f_{2}$ are two polynomials and $r = f_{1}/f_{2}$ is a rational function, then define $r^{\vee} = f_{1}^{\vee}/f_{2}^{\vee}.$ The following proposition is an intermediate step to prove the main theorem of this section.
	
	\begin{proposition}\label{D}
		Let $x = (x_{1},x_{2},x_{3})$ be a rational $\cP-$inner function.  Let $x_{1} = B \frac{f_{1}}{g_{1}}$ where $B$ is a Blaschke product and $f_{1},g_{1}$ are relatively prime polynomials such that $f_{1} / g_{1}$ has no Blaschke factor. Then the following hold.
		\begin{enumerate}
			\item If $g_{1}(a) = 0,$ then $x_{1}^{\vee} (1/a) \neq 0$; and 
			\item if $x_{2} = f_{2} / g_{2},$ where $f_{2}$ and $g_{2}$ are relatively prime polynomials, then $g_{1} = t g_{2}$ for some non-zero constant $t.$ 
		\end{enumerate}
	\end{proposition}
	\begin{proof}
		Let $x = (x_{1},x_{2},x_{3})$ be a rational $\cP-$inner function. Let $g_{1}(a) = 0$. Suppose if possible $x_{1}^{\vee}(1/a) = 0.$ This implies that
		$f^{\vee}_{1}(1/a) = 0$, which in turn implies that $f_{1}(1/\overline{a}) = 0,$ this together with $g_{1}(a) = 0,$ imply that $f_{1}/g_{1}$ has a Blaschke factor, which is a contradiction. Hence, $x_{1}^{\vee}(1/a) \neq 0.$ This proves $(1).$
		
		Since $x = (x_{1},x_{2},x_{3})$ is a rational $\cP-$inner function, $(x_{2}$ and $x_{3})$ is a $\Gamma-$inner function. Therefore, $x_{2}, x_{3}$ satisfy
		$$ x_{2}(\lambda) = \overline{x_{2}(\lambda)} x_{3}(\lambda) = x_{2}^{\vee}(\overline{\lambda}) x_{3}(\lambda) = x_{2}^{\vee}(1/\lambda)x_{3}(\lambda)$$ for all $\lambda \in \mathbb{T}.$ Since the first and last terms are rational functions,
		$$ x_{2}(\lambda) = x_{2}^{\vee}(1/\lambda) x_{3}(\lambda) \hspace{5mm} \text{for all } \lambda\in\mathbb{C}. $$ Hence,
		$$ x_{2}(a)\neq 0 \Rightarrow x_{2}^{\vee}(1/a) \neq 0.$$
		Since $x = (x_{1},x_{2},x_{3})$ is a rational $\cP-$inner function, $x_{1}$ and $x_{2}$ satisfy
		\begin{align}
			& x_{1}(\lambda) \overline{x_{1}(\lambda)} = 1 - \frac{1}{4} x_{2}(\lambda)\overline{x_{2}(\lambda)} \nonumber\\
			\Rightarrow & x_{1}(\lambda) x_{1}^{\vee}(\overline{\lambda}) = 1 - \frac{1}{4} x_{2}(\lambda) x_{2}^{\vee}(\overline{\lambda})\nonumber
		\end{align}
		 for all $\lambda \in \mathbb{T}.$ This implies
		\begin{align}\label{Eq11}
			x_{1}(\lambda) x_{1}^{\vee}(1/\lambda) = 1 - \frac{1}{4}x_{2}(\lambda) x_{2}^{\vee}(1/\lambda)\quad \text{ for all } \lambda \in \mathbb{T}.
		\end{align}
		Since both the left hand side and the right hand side are rational functions in equation \eqref{Eq11}, it follows that
		$$ x_{1}(\lambda) x_{1}^{\vee}(1/\lambda) = 1 - \frac{1}{4}x_{2}(\lambda) x_{2}^{\vee}(1/\lambda) \hspace{5mm} \text{for all } \lambda\in\mathbb{C}.$$
		For $m \geq 1$, we have
		\begin{align}\label{Eq22}
			(\lambda - a)^{m-1} x_{1}(\lambda) x_{1}^{\vee}(1/\lambda) = (\lambda - a)^{m-1}(1 - \frac{1}{4}x_{2}(\lambda) x_{2}^{\vee}(1/\lambda))
		\end{align}
		for all $\lambda\in\mathbb{C}$.
		
		Let $a$ be a pole of $x_{1}$ of multiplicity $m\geq 1.$  Clearly, $|a|>1$. Hence $|1/a|<1,$ and so $x_{1}^{\vee}, x_{2}^{\vee}$ are analytic at $1/a.$ Also by part-1 of the proposition $x_{1}^{\vee} (1/a) \neq 0.$ Therefore, on letting $\lambda \rightarrow a$ in \eqref{Eq22}, we get $$(\lambda - a)^{m-1} x_{2}(\lambda) \rightarrow \infty.$$ Thus $a$ is a pole of $x_{2}$ of multiplicity at least $m.$
		
		Let $a$ be a pole of $x_{2}$ of multiplicity $m\geq 1.$ Again on letting $\lambda \rightarrow a$ in equation \eqref{Eq22} we get that $a$ is a pole of $x_{1}$ of multiplicity at least $m.$ This proves that $g_{1}$ and $g_{2}$ have same zeros with same multiplicities. Hence $g_{1} = t g_{2}$ for some non-zero constant $t$. 
	\end{proof}

	Now we are ready to prove the main result of this section.
	
	\begin{thm}\label{Penta-inner}
		If $x=(x_1, x_2, x_3)$ is a rational $\cP-$inner function and the degree of $x_3$ is $n$, then there exist polynomials $N_1, N_2, D$ and a finite Blaschke product $B$ such that
		\begin{enumerate}
			\item $(x_{2},x_{3}) = \left(\frac{N_{2}}{D}, \frac{D^{\sim n}}{D}\right)$ is a $\Gamma-$inner function,
			\item $x_1 =B\frac {N_1}{D}$ on $\overline{\mathbb{D}}$,
			\item $|N_{1}(\lambda)|^2= |D(\lambda)|^2- \frac{1}{4}|N_2(\lambda)|^2$ on $\mathbb{T}$, and
			\item $\operatorname{deg}(N_1) \leq n$.
		\end{enumerate}
		Conversely, if $N_1, N_2,$ and $D$ are polynomials satisfying $(1)$ and $(3)$ above, then
		$(\frac{N_{1}}{D}, \frac{N_{2}}{D}, \frac{D^{\sim n}}{D})$ is a rational $\cP-$inner function and the degree of $\frac{D^{\sim n}}{D}$ is equal to $n$.
		
		Furthermore, a triple of polynomials $N_1', N_2'$ and $D'$ satisfy $(1)–(4)$ if and only if there
		exists a non-zero real number $t$ such that
		$$N_1=tN_1',\quad N_2=tN_2' \quad\text{ and } \quad D=tD'.$$
	\end{thm}
	\begin{proof}
		Let $x = (x_{1},x_{2},x_{3})$ be a rational $\cP-$inner function and the degree of $x_{3}$ be $n.$ Then $(x_{2},x_{3})$ is a rational $\Gamma-$inner function. By Theorem \ref{Gamma-Inner}, there exist two polynomials $N_{2}$ and $D$ of degree less than or equal to $n$ such that
		$$(x_{2},x_{3}) =  \left(\frac{N_{2}}{D}, \frac{D^{\sim n}}{D}\right).$$
		This proves condition $(1).$ Note that $D(\lambda) \neq 0$ for all $\lambda \in \overline{\bD}.$ Since $x_{1}$ is a rational function with poles outside $\overline{\bD},$ we have
		$$ x_{1} = B \frac{f}{g}$$ where $B$ is a finite Blaschke product and $f,g$ are relatively prime polynomials such that $f/g$ does not contain any Blaschke factor. By Proposition \ref{D}, $g$ can be taken to be $D.$ Let us denote $f$ by $N_{1}.$ Thus, $$x_{1} = B \frac{N_{1}}{D}.$$ This proves condition $(2).$
		
		Since $\big(x_{1}(\lambda),x_{2}(\lambda),x_{3}(\lambda)\big)\in b \cP$ for all $\lambda\in\mathbb{T},$ we have
		$$|x_{1}(\lambda)|^{2} = 1 - \frac{1}{4}|x_{2}(\lambda)|^{2}.$$
		By virtue of conditions $(1)$ and $(2),$ we have
		\begin{align}
			& \bigg|\frac{N_{1}(\lambda)}{D(\lambda)}\bigg|^{2} = 1 - \frac{1}{4} \bigg|\frac{N_{2}(\lambda)}{D(\lambda)}\bigg|^{2} \nonumber \\
			\Rightarrow &  |N_{1}(\lambda)|^{2} = |D(\lambda)|^{2} - \frac{1}{4}|N_{2}(\lambda)|^{2} \label{Eq33}
		\end{align}	
		for all $\lambda\in\mathbb{T}.$ This proves condition $(3).$
		
		From equation \eqref{Eq33}, it follows that
		$$N_{1}(\lambda) N_{1}^{\vee}(\overline{\lambda}) = D(\lambda) D^{\vee}(\overline{\lambda}) -\frac{1}{4} N_{2}(\lambda) N_{2}^{\vee}(\overline{\lambda}).$$
		This is same as
		\begin{align}
			N_{1}(\lambda) N_{1}^{\vee}(1/\lambda) = D(\lambda) D^{\vee}(1/\lambda) - \frac{1}{4}N_{2}(\lambda) N_{2}^{\vee}(1/\lambda) \label{Eq44}
		\end{align}
		for all $\lambda\in\mathbb{T}.$ Since $N_{1}(0) \neq 0,$ the coefficient of $\lambda^{\operatorname{deg}(N_{1})}$ is non-zero in $N_{1}(\lambda) N_{1}^{\vee}(1/\lambda),$ which is the highest degree coefficient in this expression. Since the degree of the right hand side in equation \eqref{Eq44} is at most $n,$ we get $\operatorname{deg}(N_{1}) \leq n.$ This proves condition $(4).$
		
		Proof of the converse follows from Theorem \ref{Gamma-Inner} and Lemma \ref{suff}.
		
		Finally, suppose a triple of polynomials $N_{1}', N_{2}'$ and $D'$ satisfy $(1) - (4).$ By Theorem \ref{Gamma-Inner}, there exists a non-zero real number $t$ such that $N_{2} = t N_{2}' $ and $D = t D'.$ Using $(2)$ we get $N_{1} = t N_{1}'. $ The converse is straightforward.
	\end{proof}

	\vspace{0.1in} \noindent\textbf{Acknowledgement:}
	
	The research works of the first author supported by the Prime Minister Research Fellowship PM/MHRD-20-15227.03. The authors thank Prof. Tirthankar Bhattacharyya for his valuable discussions and suggestions. We are greatful to the anonymous referee for useful comments and suggestions, in particular on Theorem \ref{dist} and Propostion \ref{D}.

\end{document}